\documentclass[reqno,12pt]{amsart}
\usepackage[left=1.2in, right = 1.2in, top=1in]{geometry}
\usepackage{amsmath, amssymb}
\usepackage{amsthm}
\usepackage{tikz}
\usepackage{amstext}
\usepackage{pifont}
\usepackage{appendix}
\usepackage{verbatim}
\usepackage{cases}
\usepackage{CJKutf8}
\newtheorem{theorem}{Theorem}[section]

\newtheorem{lemma}{Lemma}[section]

\newtheorem{claim}{claim}

\usepackage{fancyhdr}

\title{Liouville theorem for quasilinear elliptic equations %involving  gradient term
in $\mathbb R^N$}

\author{Wangzhe Wu}
\address{Department of Mathematics \\University of Science and Technology of China, Hefei, China}
\email{wuwz18@mail.ustc.edu.cn}

\author{Qiqi Zhang}
\address{College of Mathematics and Information Science \\ Nanchang Hangkong University, Nanchang, China}
\email{qiqizmath@126.com}

\begin{document}
	\begin{CJK}{UTF8}{gbsn}
	\pagestyle{fancy}

\fancyhead{}
%\fancyhead[CO]{Liouville theorem for elliptic equations}

\fancyhead[CE]{\leftmark}

	\begin{abstract}
		We prove Liouville theorem for the equation $\Delta_m v +  v^p + M |\nabla v|^{q}=  0$ in  a domain $\Omega\subset\mathbb R^n$, with $M\in \mathbb{R}$ in the critical and subcritical case. As a natural extension of our recent work \cite{MWZ}, the proof is based on an integral identity and Young's inequality.
	\end{abstract}

	\maketitle

	\section{Introduction}

\noindent In this paper, we are concerned with a Liouville type theorem for the global solutions  of the following equations in $\mathbb R^n$:
	\begin{equation}\label{main equation}
		\Delta_m v +  v^p + M |\nabla v|^{q}=0, \quad \mathrm{ in} \quad\Omega,
	\end{equation}
where $q = \frac{mp}{p + 1}$,  $\Delta_m v = \mathrm{div}(|\nabla v| ^{ m-2}\nabla v )$, $p$ and $q$ are exponents larger than 1, $M\in \mathbb{R}$.  This kind of equations are widely concerned and studied depending on the sign and value of $M$.
The most important motivation of the present study is to extend the results obtained for the semilinear equation
\begin{equation}\label{laplacian}
\Delta v +  v^p + M |\nabla v|^{q}=0,\quad \mathrm{ in} \quad\Omega,
\end{equation}
in \cite{VERON,VERON0}. By using a delicate combination of refined Bernstein techniques and Keller-Osserman estimate, they obtained a series of a priori estimates for any positive solution of (\ref{laplacian}) in arbitrary domain of $\mathbb R^n$.
\begin{theorem}[Theorem E in \cite{VERON}] \label{Veron0}
		Let $n \geq 3$, $1 < p < \frac{n + 2}{n - 2} , q = \frac{2p}{p + 1}$. Then there exists $\epsilon_0 > 0$ depending on $n$ and $p$ such that for any $|M| \leq \epsilon_0$,
		there exists no nontrivial nonnegative solution of \eqref{laplacian} in $\mathbb R^n$.
	\end{theorem}
It's obviously to see that the result in \cite{VERON} did not cover all the cases that $m = 2, 1 < p \leq \frac{n + 2}{n - 2}$ for $q = \frac{2p}{p + 1}$ when $M > 0$.  Thus, our recent work \cite{MWZ} came out, which complete the corresponding results of the range of $p$ with $1 < p \le \frac{n + 2}{n - 2}$ when $M > 0$, based on an integral identity and Young's inequality. For $m = 2, M < 0$, in \cite{VERON0} they gave some Liouville theorems for the equation \eqref{laplacian} when $v$ is radial. In addition, for general case $1 < m < p$, there are also parts of results varies as $M$ changes.
If $M>0$, Filippucci-Sun-Zheng \cite{F-Y-Y} derived various a priori estimate concerning $|\nabla v|$ for positive solutions of (\ref{main equation}) in the cases $q$ is less, greater or equal to ${mp\over p+1}$, and consequently obtain Liouville type theorems. For the case $q=\frac{mp}{p+1}$ and $M$ large enough, they got the following non-existence result in $\mathbb R^n$.
\begin{theorem}[Theorem 1.4  in \cite{F-Y-Y}]\label{m-lapalace Theorem}
		Let $\Omega\subset \mathbb R^n$, $p>\max\{m-1,1\}$ and $q=\frac{mp}{p+1}$. Then for any
 $$
 M>\frac{\sqrt{n}(p+1)}{(4p)^{p\over p+1}}\Big({p-1\over \sqrt{a}}\Big)^{\frac{p-1}{p+1}},
 $$
 where $0 < a \leq 1/n$, (\ref{main equation}) does not admit positive solutions in $\mathbb R^n$.
	\end{theorem}

If $M<0$, especially $M=-1$, which has been introduced by Chipot-Weissler \cite{Chipot-Weissler} in 1989 and obtained the existence of ground states when the parameter values $p$ and $q$ satisfy $1<q<{2p\over p+1}$, $p>1$. we recall that a ground states for (\ref{main equation}) or (\ref{laplacian}) is a non-negative non-trivial entire solution. It's known that radial ground states for (\ref{laplacian}) always exist for the supercritical parameter range $p > (n+ 2 ) / ( n - 2 )$ and may or may not exist when $p$ is critical or subcritical, depending on the value of $q$. Turning to the case of ground states for equation (\ref{main equation}) with $M=-1$, Serrin-Zou \cite{Serrin-Zou} showed that existence always holds when $p$ is supercritical, that is $p>\frac{(m-1)n+m}{n-m}$. For critical $p$, existence of radial ground states holds if and only if
\begin{align*}
\begin{cases}
0<q<(m-1)\frac{(m-1)n+m}{(m-1)n-m}, \text{ if } m >\frac{n}{n-1} ; \\
q>0, \text{ if } m <\frac{n}{n-1}.
\end{cases}
\end{align*}
For subcritical $p$, existence holds when $q<\frac{mp}{p+1}$.
Besides, when $M$ is allowed to be other negative constant, Filippucci-Sun-Zheng \cite{F-Y-Y} derived a non-existence result for supersolutions of (\ref{main equation}) in an exterior domain.

\begin{theorem} [Theorem  1.5 in \cite{F-Y-Y}]
Let $p > m-1$ if $n=m$ or $m-1<p<\frac{n(m-1)}{n-m}$ if $n > m$, $q=\frac{mp}{p+1}$ and $ M>-\nu^*(n)$ where
 $$
 \nu^*(n):=(p+1) \Big( \frac{n(m-1)-p(n-m)}{mp} \Big )^ {p\over p+1} .
 $$
Then there exist no nontrivial nonnegative supersolutions of  (\ref{main equation}) in $\mathbb R^n\setminus \bar{B}_R$ for any $R>0$.
	\end{theorem}

If $M=0$, then (\ref{main equation}) reduces to the generalized Lane-Emden equation

\begin{equation}\label{Lane-Emden}
		\Delta_m v +  v^p =0, \quad \mathrm{ in} \quad\Omega,
	\end{equation}
which has been widely studied in the literature \cite{Bidaut-Veron,Bidaut-Veron0,Bidaut-Veron1,chen,DMMS,Gidas,Guedda,Ni,Ni1,Serrin,Serrin1,Vetois}, both when $\Omega$ is bounded and  unbounded. Especially, in the semilinear case $m=2$, one of the celebrated results is given by Gidas and Spruck \cite{Gidas}:
if $n>2$, and $p\in \Big[1,\frac{n+2}{n-2}\Big)$, then any nonnegative solution of (\ref{Lane-Emden}) is identically zero and the result is sharp.
For the case of $m > 1$, radially symmetric positive solutions were studied by Ni and Serrin \cite{Ni0,Ni,Ni1}, and further results in this direction were obtained by Guedda and V\'eron \cite{Guedda} and Bidaut-V\'eron \cite{Bidaut-Veron}. When one studies the so called Liouville property of (\ref{Lane-Emden}), namely whether all positive $C^1$ solutions of (\ref{Lane-Emden}) in $\mathbb R^n$ are constant, two critical exponents appear
$$
m_*=\frac{n(m-1)}{n-m},\quad m^*=\frac{n(m-1)+m}{n-m},
$$
where $n>m$, known as the Serrin exponent and the Sobolev exponent, respectively.

For the case only with gradient terms $(M\neq0)$, we first recall the Hamilton-Jacobi equation
\begin{equation}\label{Hamilton-Jacobi}
		\Delta_m v +  |\nabla v|^{q} =0, \quad \mathrm{ in} \quad\Omega,
	\end{equation}
The Liouville property of (\ref{Hamilton-Jacobi}) was studied by Lions in \cite{P-L} for special case $m =2$, who proved that any $C^2$ solution to (\ref{Hamilton-Jacobi}) with $q > 1$ has to be a constant by using the Bernstein technique. In \cite{Chipot-Weissler} and \cite{SZ}, most of the study deals with the case $q \neq \frac{2p}{p + 1}$. In the critical case, then not only the sign of $M$ but also its value plays a fundamental role, with a delicate interaction with the exponent $p$. Other work on this kind of equation can be seen in \cite{Bidaut-Veron3}. For general case, Bidaut-V\'eron, Garcia-Huidobro and V\'eron \cite{Bidaut-Veron3} proved that any $C^1$ solution of (\ref{Hamilton-Jacobi}) in an arbitrary domain $\Omega$ with $n \ge m > 1$ and $q > m-1$ satisfies
$$
|\nabla v(x)|\le c_{n,m,q}\big(\mathrm {dist}(x,\partial \Omega)\big)^{-\frac{1}{q-m+1}}.
$$
Estimates of this type, not only for the gradient but also for the solutions are called by Serrin and Zou "universal a priori estimates", because they are independent of the solutions and do not need any boundary conditions.

In this paper, we continue the idea used in our recent work \cite{MWZ} to derive Liouville type theorems for positive solutions of (\ref{main equation}) in the case $q=\frac{mp}{p+1}$ varies in the sign of $M$, that produces the following results.
\begin{theorem}
		Let $n \geq 2$, $m - 1 < p \leq \frac{(m - 1)n + m}{n - m}$, $q = \frac{mp}{p + 1}$, then there exists $M_1 > 0$, which depends on $n, m, p$, such that for any $0 < M < M_1$, all the nonnegative solutions of \eqref{main equation} are $v \equiv 0$.
	\end{theorem}
\begin{theorem}
		Let $n \geq 2$, $m - 1 < p < \frac{(m - 1)n + m}{n - m}$, $q = \frac{mp}{p + 1}$, then there exists $M_2 > 0$, which depends on $n, m, p$, such that for any $-M_2 < M < 0$, all the nonnegative solutions of \eqref{main equation} are $v \equiv 0$.
	\end{theorem}
	The second theorem extends the result in \cite{F-Y-Y} from $m-1<p<\frac{n(m-1)}{n-m}$ to $m - 1 < p < \frac{(m - 1)n + m}{n - m}$. But when $m-1<p<\frac{n(m-1)}{n-m}$, we don't know whether $M_2 > \nu^*(n)$.

	\section{Integral identity}\label{1}
	%Consider the equation:
%	\begin{equation}
%		\Delta_m v +  v^p + M |\nabla v|^{q}=0.
%	\end{equation}

\noindent Multiplying (\ref{main equation}) by $v^{\alpha}\Delta_m v$,
\begin{equation}\label{2}
		\begin{aligned}
			   &\underset{\textcircled{1}}{ v^\alpha(\Delta_m v)^2 } =\underset{I}{ - v^{\alpha + p } \Delta_m v }  \underset{II} { - M v^{\alpha}|\nabla v| ^{ q}\Delta_m v. }
		\end{aligned}
	\end{equation}
	We observed that the term $\textcircled{1}$ is important. Denote
	%In the following contents, we always suppose that $v_1(x_0) = |\nabla v|(x_0)$ and define
	\begin{equation*}
		\begin{aligned}
			&X^i =|\nabla v| ^{ m-2}v_i,\\
            &\Delta_m v=(|\nabla v| ^{ m-2}v_i)_i,\\
            &X^i_{j}=(|\nabla v| ^{ m-2}v_i)_j,\\
			&E^i_{j} = X^i_{j} - \frac{\Delta_m v}{n}\delta_{ij}, \\
		\end{aligned}
	\end{equation*}
then we will do some computations on some important terms as follows.
	%and
%	\begin{equation*}
%		E_{ij} = v_{ij} - \frac{1}{n}\delta_{ij}\Delta v, \text{ for }  i, j = 1,...,n.
%	\end{equation*}
%	We remark that the definition of $E_{ij}$ is global but $G_{ij}$ is not.

	\subsection{The term $v^{\alpha - 1}|\nabla v|^{m}\Delta _m v $}
	
	\begin{equation}\label{section1_equ5}
		\begin{aligned}
			v^{\alpha - 1}|\nabla v|^{m}\Delta_m v &= v^{\alpha - 1}|\nabla v|^m\big( |\nabla v| ^{ m-2}v_i\big)_i \\
             &=(v^{\alpha - 1} |\nabla v| ^{ 2m-2}v_i)_i
- (\alpha - 1)v^{\alpha - 2}|\nabla v|^{2m} - m \underset{III}{v^{\alpha - 1}|\nabla v|^{2m-4}v_i v_j v_{ij} },\\
		\end{aligned}
	\end{equation}
here
	\begin{equation}
	\begin{aligned}
     III&=v^{\alpha - 1}|\nabla v|^{2m-4}v_i v_j v_{ij}\\
     &=v^{\alpha - 1}|\nabla v|^{m-2}v_i v_j(|\nabla v|^{m-2} v_{ij})\\
     &=v^{\alpha - 1}|\nabla v|^{m-2}v_i v_j\Big[(|\nabla v| ^{ m-2}v_i)_j-(m-2)|\nabla v|^{m-4}v_i v_\ell v_{\ell j} \Big]\\
     &=v^{\alpha - 1}|\nabla v|^{m-2}v_i v_j X^i_{j}-(m-2)v^{\alpha - 1}|\nabla v|^{2m-4}v_i v_j v_{ij},
    \end{aligned}
	\end{equation}
therefore,
\begin{equation}\label{section1_equ0}
	\begin{aligned}
 III&={1\over m-1}v^{\alpha - 1}|\nabla v|^{m-2}v_i v_j X^i_{j}\\
 &={1\over m-1}v^{\alpha - 1}|\nabla v|^{m-2}v_i v_j(E^i_{j}+ \frac{\Delta_m v}{n}\delta_{ij})\\
 &={1\over m-1}v^{\alpha - 1}|\nabla v|^{m-2}v_i v_jE^i_{j}+{1\over n(m-1)}v^{\alpha - 1}|\nabla v|^{m}\Delta_m v.
 \end{aligned}
	\end{equation}
Substitute (\ref{section1_equ0}) into (\ref{section1_equ5}), we obtain

\begin{equation}\label{section1_equ50}
	\begin{aligned}
&v^{\alpha - 1}|\nabla v|^{m}\Delta_m v \\
=&\frac{n(m-1)}{n(m-1)+m}\cdot \Big[ \big(v^{\alpha - 1} |\nabla v| ^{ 2m-2}v_i\big)_i-(\alpha - 1)v^{\alpha - 2}|\nabla v|^{2m}-{m\over m-1}v^{\alpha - 1}|\nabla v|^{m-2}v_i v_jE^i_{j} \Big].
\end{aligned}
	\end{equation}

	\subsection{The term \textcircled{1}}\label{The term 1}
	
	In fact, we have
	\begin{equation*}
		\begin{aligned}
			v^{\alpha}(\Delta_m v)^2 & = v^{\alpha} X^i_{i} X^j_{j}\\
			&= (v^{\alpha} X^i_{i} X^j)_j - \alpha v^{\alpha - 1} X^i_{i} X^j v_{j} - v^{\alpha}  X^j (X^i_{i})_j \\
			&=(v^{\alpha} X^i_{i} X^j)_j - \alpha v^{\alpha - 1} |\nabla v|^{m}\Delta_m v-\Big[ (v^{\alpha}  X^j X^i_{j})_i-(v^{\alpha}  X^j)_i X^i_{j}\Big]\\
&=(v^{\alpha} X^i_{i} X^j)_j - \alpha v^{\alpha - 1} |\nabla v|^{m}\Delta_m v-\Big[ (v^{\alpha}  X^j X^i_{j})_i-\alpha v^{\alpha-1}v_i  X^j X^i_{j}- v^{\alpha} X^i_{j}X^j_{i}\Big]\\
&=\Big( v^{\alpha}\Delta_m v X^i-v^{\alpha}  X^j X^i_{j}\Big)_i-\alpha v^{\alpha - 1} |\nabla v|^{m}\Delta_m v+\alpha v^{\alpha-1}v_i  X^j X^i_{j}+ v^{\alpha} X^i_{j}X^j_{i}.
		\end{aligned}
	\end{equation*}
Since $ X^i_{j}=E^i_{j} + \frac{\Delta_m v}{n}\delta_{ij} $,
then
\begin{equation*}
		\begin{aligned}
			v^{\alpha}(\Delta_m v)^2 & =(1-{1\over n}) (v^{\alpha}\Delta_m v X^i)_i-(v^{\alpha}E^i_{j} X^j )_i-\alpha v^{\alpha - 1} |\nabla v|^{m}\Delta_m v\\
&+\alpha v^{\alpha-1}v_i  X^j E^i_{j}+{\alpha\over n}v^{\alpha - 1} |\nabla v|^{m}\Delta_m v+v^{\alpha} E^i_{j}E^j_{i}+{1\over n}v^{\alpha}(\Delta_m v)^2,
\end{aligned}
	\end{equation*}
$\Rightarrow$
\begin{equation*}
		\begin{aligned}
(1-{1\over n})v^{\alpha}(\Delta_m v)^2 & =(1-{1\over n}) (v^{\alpha}\Delta_m v X^i)_i-(v^{\alpha}E^i_{j} X^j )_i-(1-{1\over n})\alpha v^{\alpha - 1} |\nabla v|^{m}\Delta_m v\\
&+\alpha v^{\alpha-1}v_i  X^j E^i_{j}+v^{\alpha} E^i_{j}E^j_{i}.
\end{aligned}
	\end{equation*}

As a result, we obtain
\begin{equation}\label{delta2}
		\begin{aligned}
			\textcircled{1} &=v^{\alpha}(\Delta_m v)^2\\
 & =(v^{\alpha}\Delta_m v X^i)_i-{n\over n-1}(v^{\alpha}E^i_{j} X^j )_i-\alpha v^{\alpha - 1} |\nabla v|^{m}\Delta_m v+{n\over n-1}\alpha v^{\alpha - 1} |\nabla v|^{m-2}E^i_{j}v_i v_j+{n\over n-1}v^{\alpha} E^i_{j}E^j_{i}\\
&= (v^{\alpha}\Delta_m v X^i)_i-{n\over n-1}(v^{\alpha}E^i_{j} X^j )_i +{n\over n-1}v^{\alpha} E^i_{j}E^j_{i}+{n\over n-1}\alpha v^{\alpha - 1} |\nabla v|^{m-2}E^i_{j}v_i v_j\\
&-\alpha\frac{n(m-1)}{n(m-1)+m}\Big[ (v^{\alpha - 1} |\nabla v| ^{ 2m-2}v_i)_i-(\alpha - 1)v^{\alpha - 2}|\nabla v|^{2m}-{m\over m-1}v^{\alpha - 1}|\nabla v|^{m-2}v_i v_jE^i_{j} \Big]\\
&=\big(v^{\alpha}\Delta_m v X^i \big)_i-{n\over n-1}\big(v^{\alpha}E^i_{j} X^j \big)_i-\alpha\frac{n(m-1)}{n(m-1)+m}\Big(v^{\alpha - 1} |\nabla v| ^{ 2m-2}v_i\Big)_i+{n\over n-1}v^{\alpha} E^i_{j}E^j_{i}\\
&+\alpha(\alpha - 1)\frac{n(m-1)}{n(m-1)+m}v^{\alpha - 2}|\nabla v|^{2m}\\
&+\alpha\Big[{n\over n-1}+\frac{nm}{n(m-1)+m}   \Big]v^{\alpha - 1} |\nabla v|^{m-2}E^i_{j}v_i v_j.
\end{aligned}
	\end{equation}

	\subsection{The term $A$}
	
	\begin{equation*}
		\begin{aligned}
			A &= (v^{\alpha}X^i \Delta_m v)_i\\
			&= \alpha v^{\alpha - 1} |\nabla v|^m \Delta_m v  + v^{\alpha}(\Delta_m v)^2+ \underset{A_1}{ v^{\alpha}X^i (\Delta_m v)_i }.\\
		\end{aligned}
	\end{equation*}
	Substitute (\ref{main equation}) into $A_1$,
	\begin{equation*}
		\begin{aligned}
			A_1 &=  v^{\alpha}X^i \Big( -v^{p} - M|\nabla v|^q\Big)_i\\
			&= -pv^{\alpha + p - 1}|\nabla v|^{m} - qM v^{\alpha}|\nabla v|^{ m+q - 4}v_i v_j v_{ij}.
		\end{aligned}
	\end{equation*}	
	Multiply (\ref{main equation})  by $v^{\alpha - 1}|\nabla v|^{m}$
	\begin{equation}
		\begin{aligned}
			- v^{\alpha + p - 1}|\nabla v|^{m} = v^{\alpha - 1}|\nabla v|^{m}\Delta_m v  + M v^{\alpha - 1}|\nabla v|^{m + q}.
		\end{aligned}
	\end{equation}
	Therefore, we have
	\begin{equation}\label{sub2_2}
		\begin{aligned}
			A_1 &=  p v^{\alpha - 1}|\nabla v|^{m}\Delta_m v  + p M v^{\alpha - 1}|\nabla v|^{m + q}- qM v^{\alpha}|\nabla v|^{m + q - 4}v_i v_j v_{ij}.\\
		\end{aligned}
	\end{equation}
On the other hand, we know
\begin{equation}\label{sub2_1}
		\begin{aligned}
			  - qM v^{\alpha }|\nabla v|^{m+ q - 4}v_i v_j v_{ij}&=   - Mv^{\alpha }(|\nabla v|^q)_i |\nabla v|^{m-2} v_i\\
             &=  - Mv^{\alpha }(|\nabla v|^q)_i X^i\\
             &= - M\Big[ (v^{\alpha }|\nabla v|^q X^i)_i -|\nabla v|^q (v^\alpha X^i)_i\Big]\\
             &=  - M(v^{\alpha }|\nabla v|^q X^i)_i + \alpha M v^{\alpha - 1}|\nabla v|^{m+ q}+Mv^{\alpha }|\nabla v|^q\Delta_m v\\
			 &= -  M(v^{\alpha }|\nabla v|^q X^i)_i +  \alpha M v^{\alpha - 1}|\nabla v|^{m + q }  +  Mv^{\alpha}|\nabla v|^{ q}\Big( - v^{p} - M |\nabla v|^q \Big)\\
			 &= -M(v^{\alpha }|\nabla v|^q X^i)_i +  \alpha M v^{\alpha - 1}|\nabla v|^{m + q } - Mv^{\alpha + p}|\nabla v|^ q - M^2 v^{\alpha}|\nabla v|^{2q}.
		\end{aligned}
	\end{equation}
$\Rightarrow$
the last two terms of (\ref{sub2_2}) is
	\begin{equation}\label{sub2_1}
		\begin{aligned}
			 & pMv^{\alpha - 1}|\nabla v|^{m+ q } - qM v^{\alpha }|\nabla v|^{m+ q - 4}v_i v_j v_{ij}\\
			 &= -M(v^{\alpha }|\nabla v|^q X^i)_i +\Big( p + \alpha\Big) Mv^{\alpha - 1}|\nabla v|^{m + q } - Mv^{\alpha + p}|\nabla v|^ q - M^2 v^{\alpha}|\nabla v|^{2q}.
		\end{aligned}
	\end{equation}
	So we get that $A_1$ and
	\begin{equation*}
		\begin{aligned}
			A &= (v^{\alpha}X^i \Delta_m v)_i\\
			&= (p+\alpha) v^{\alpha - 1} |\nabla v|^m \Delta_m v  + v^{\alpha}(\Delta_m v)^2\\
			&-M(v^{\alpha }|\nabla v|^q X^i)_i +\Big( p + \alpha\Big) Mv^{\alpha - 1}|\nabla v|^{m + q } - Mv^{\alpha + p}|\nabla v|^ q - M^2 v^{\alpha}|\nabla v|^{2q}.
		\end{aligned}
	\end{equation*}
Substitute \eqref{section1_equ50} into it:
	\begin{equation}\label{delta1}
		\begin{aligned}
			  &-v^{\alpha}(\Delta_m v)^2\\
			& = -A + (\alpha + p ) \cdot\frac{n(m-1)}{n(m-1)+m}\Big[ (v^{\alpha - 1} |\nabla v| ^{ 2m-2}v_i)_i-(\alpha - 1)v^{\alpha - 2}|\nabla v|^{2m} \\
			&-{m\over m-1}v^{\alpha - 1}|\nabla v|^{m-2}v_i v_jE^i_{j} \Big] \\
			&- M(v^{\alpha }|\nabla v|^q X^i)_i + \Big( p + \alpha\Big) Mv^{\alpha - 1}|\nabla v|^{m + q } - Mv^{\alpha + p}|\nabla v|^ q - M^2 v^{\alpha}|\nabla v|^{2q}.
		\end{aligned}
	\end{equation}

	\subsection{The final equation}

\medskip
Combine (\ref{delta2}) and  (\ref{delta1}) to delete $v^{\alpha}(\Delta_m v)^2$, we have	
	\begin{align*}
		0&= -A + (\alpha + p ) \cdot\frac{n(m-1)}{n(m-1)+m}\Big[ (v^{\alpha - 1} |\nabla v| ^{ 2m-2}v_i)_i-(\alpha - 1)v^{\alpha - 2}|\nabla v|^{2m} \\
		&-{m\over m-1}v^{\alpha - 1}|\nabla v|^{m-2}v_i v_jE^i_{j} \Big] \\
			&- M(v^{\alpha }|\nabla v|^q X^i)_i + \Big( p + \alpha\Big) Mv^{\alpha - 1}|\nabla v|^{m + q } - Mv^{\alpha + p}|\nabla v|^ q - M^2 v^{\alpha}|\nabla v|^{2q}\\
			&+\big(v^{\alpha}\Delta_m v X^i \big)_i-{n\over n-1}\big(v^{\alpha}E^i_{j} X^j \big)_i-\alpha\frac{n(m-1)}{n(m-1)+m}\Big(v^{\alpha - 1} |\nabla v| ^{ 2m-2}v_i\Big)_i+{n\over n-1}v^{\alpha} E^i_{j}E^j_{i}\\
&+\alpha(\alpha - 1)\frac{n(m-1)}{n(m-1)+m}v^{\alpha - 2}|\nabla v|^{2m}\\
&+\alpha\Big[{n\over n-1}+\frac{nm}{n(m-1)+m}   \Big]v^{\alpha - 1} |\nabla v|^{m-2}E^i_{j}v_i v_j,
	\end{align*}
	
that is
	\begin{equation*}
		\begin{aligned}
			0& = -A + (\alpha + p ) \cdot\frac{n(m-1)}{n(m-1)+m} (v^{\alpha - 1} |\nabla v| ^{ 2m-2}v_i)_i -  M(v^{\alpha }|\nabla v|^q X^i)_i \\
			&+(v^{\alpha}\Delta_m v X^i)_i-{n\over n-1}(v^{\alpha}E^i_{j} X^j )_i-\alpha\frac{n(m-1)}{n(m-1)+m}\Big(v^{\alpha - 1} |\nabla v| ^{ 2m-2}v_i\Big)_i\\
			 &+{n\over n-1}v^{\alpha} E^i_{j}E^j_{i} -p(\alpha-1)\frac{n(m-1)}{n(m-1)+m}v^{\alpha - 2}|\nabla v|^{2m}\\
			 & + \Big[{n\alpha\over n-1}- \frac{npm}{n(m-1)+m}\Big]v^{\alpha - 1}|\nabla v|^{m-2}v_i v_jE^i_{j}\\
			& + \Big( p + \alpha\Big) Mv^{\alpha - 1}|\nabla v|^{m + q } -Mv^{\alpha + p}|\nabla v|^ q - M^2 v^{\alpha}|\nabla v|^{2q}.
		\end{aligned}
	\end{equation*}
Now we can integrate it into the following form:	
	\begin{equation}\label{sec1_final}
		\begin{aligned}
		    & 0 = W + {n\over n-1}v^{\alpha} E^i_{j}E^j_{i} - p(\alpha-1)\frac{n(m-1)}{n(m-1)+m}v^{\alpha - 2}|\nabla v|^{2m}\\
		    & +\Big[{n\alpha\over n-1} - \frac{npm}{n(m-1)+m}\Big]v^{\alpha - 1}|\nabla v|^{m-2}v_i v_jE^i_{j} +\Big( p + \alpha\Big) Mv^{\alpha - 1}|\nabla v|^{m + q } \\
			&-Mv^{\alpha + p}|\nabla v|^ q - M^2 v^{\alpha}|\nabla v|^{2q},\\
		\end{aligned}
	\end{equation}
	where $W$ consists of all the divergence terms.

	\section{Conditions}
	In order to obtain the Liouville  theorem, we need some extra conditions and we also insert coefficients $T,U,P$ for more detailed computations.

	\noindent Multiply (\ref{main equation})  by $v^{\alpha+p}$:
\begin{equation}
		\begin{aligned}
			 v^{\alpha + 2p} +  M v^{\alpha+p}|\nabla v|^{q} &= -v^{\alpha + p}\Delta_m v\\
	 		&= -v^{\alpha + p}(|\nabla v|^{m-2}v_i)_i\\
            &=-(v^{\alpha + p}|\nabla v|^{m-2}v_i)_i+ (\alpha+p)v^{\alpha + p-1}|\nabla v|^m.
		\end{aligned}
	\end{equation}

\noindent	Multiply (\ref{main equation})  by $ v^{\alpha - 1}|\nabla v|^{m}$:
	\begin{equation*}
		\begin{aligned}
			&- M v^{\alpha - 1}|\nabla v|^{m+ q } - v^{\alpha + p-1}|\nabla v|^m\\
			 &= v^{\alpha - 1}|\nabla v|^m \Delta_m v\\
             &= \frac{n(m-1)}{n(m-1)+m} \Big[ (v^{\alpha - 1} |\nabla v| ^{ 2m-2}v_i)_i-(\alpha - 1)v^{\alpha - 2}|\nabla v|^{2m} - {m\over m-1}v^{\alpha - 1}|\nabla v|^{m-2}v_i v_jE^i_{j} \Big].
		\end{aligned}
	\end{equation*}

\noindent Multiply (\ref{main equation})  by $M v^{\alpha}|\nabla v|^{ q}$:
	\begin{equation}\label{01}
		\begin{aligned}
			 M  v^{\alpha + p}|\nabla v|^{q} +  M^2 v^{\alpha}|\nabla v|^{ 2q}&=-Mv^{\alpha}|\nabla v|^{ q}\Delta_m v,\\
		\end{aligned}
	\end{equation}
as we can see that
\begin{equation}\label{X}
		\begin{aligned}
X^i_j v_i v_j&=(|\nabla v|^{m-2}v_i)_jv_i v_j\\
&=( m-2)|\nabla v|^{m-4}v_{\ell j}v_\ell v_i v_i v_j+|\nabla v|^{m-2}v_{ij} v_i v_j\\
&=(m-1)|\nabla v|^{m-2}v_{ij} v_i v_j,
\end{aligned}
	\end{equation}
if we insert (\ref{X}) into (\ref{01}), we obtain
\begin{align*}
	v^{\alpha}|\nabla v|^{ q}\Delta_m v &= \big( v^{\alpha}|\nabla v|^{ q} X^i\big)_i - \alpha  v^{\alpha-1}|\nabla v|^{ m+q} - q v^{\alpha}|\nabla v|^{ m+q-4}v_{ij}v_iv_j\\
	&= \big( v^{\alpha}|\nabla v|^{ q} X^i\big)_i - \alpha  v^{\alpha-1}|\nabla v|^{ m+q} - \frac{ q}{m - 1} v^{\alpha}|\nabla v|^{ q - 2} X^i_j v_i v_j\\
	&= \big( v^{\alpha}|\nabla v|^{ q} X^i\big)_i - \alpha  v^{\alpha-1}|\nabla v|^{ m+q} - \frac{ q}{m - 1} v^{\alpha}|\nabla v|^{ q - 2} E^i_j v_i v_j - \frac{q}{n(m - 1)}v^{\alpha}|\nabla v|^q \Delta_m v,
\end{align*}

\begin{equation}
	\begin{aligned}
		\Rightarrow v^{\alpha}|\nabla v|^{ q}\Delta_m v &= \frac{n(m - 1)}{n(m - 1) + q} \big( v^{\alpha}|\nabla v|^{ q} X^i\big)_i - \frac{n(m - 1)\alpha}{n(m - 1) + q} v^{\alpha-1}|\nabla v|^{ m+q} \\
		&- \frac{nq}{n(m - 1) + q} v^{\alpha}|\nabla v|^{ q - 2} E^i_j v_i v_j,
	\end{aligned}
\end{equation}

so (\ref{01}) turns to be
\begin{equation}
		\begin{aligned}
			 &Mv^{\alpha + p}|\nabla v|^{q} +  M^2 v^{\alpha}|\nabla v|^{ 2q} \\
			 &=-M\frac{n(m - 1)}{n(m - 1) + q} \big( v^{\alpha}|\nabla v|^{ q} X^i\big)_i + \frac{n(m - 1)\alpha}{n(m - 1) + q} M v^{\alpha-1}|\nabla v|^{ m+q} \\
		&+ \frac{nq}{n(m - 1) + q} M v^{\alpha}|\nabla v|^{ q - 2} E^i_j v_i v_j.
\end{aligned}
	\end{equation}
After that, we now rewrite (\ref{sec1_final}) as follows
\begin{equation}
		\begin{aligned}
		 & 0 = W +  {n\over n-1}v^{\alpha} E^i_{j}E^j_{i} -p(\alpha-1)\frac{n(m-1)}{n(m-1)+m}v^{\alpha - 2}|\nabla v|^{2m} \\
		&+\Big[{n\alpha\over n-1}- \frac{npm}{n(m-1)+m}\Big]v^{\alpha - 1}|\nabla v|^{m-2}v_i v_jE^i_{j} \\
		&+\Big( p + \alpha\Big) Mv^{\alpha - 1}|\nabla v|^{m + q } -Mv^{\alpha + p}|\nabla v|^ q - M^2 v^{\alpha}|\nabla v|^{2q}\\
		&+T \Bigg[ v^{\alpha + 2p} +  M v^{\alpha+p}|\nabla v|^{q} - (\alpha+p)v^{\alpha + p-1}|\nabla v|^m \Bigg]\\
		&+ U \Bigg[ Mv^{\alpha + p}|\nabla v|^{q} +  M^2 v^{\alpha}|\nabla v|^{ 2q} - \frac{n(m - 1)\alpha}{n(m - 1) + q} M v^{\alpha-1}|\nabla v|^{ m+q} \\
		&- \frac{nq}{n(m - 1) + q} M v^{\alpha}|\nabla v|^{ q - 2} E^i_j v_i v_j \Bigg]\\
		&+P \Bigg[ - M v^{\alpha - 1}|\nabla v|^{m+ q } - v^{\alpha + p-1}|\nabla v|^m + \frac{n(m - 1)(\alpha - 1)}{n(m - 1) + m} v^{\alpha - 2}|\nabla v|^{2m} \\
		&+ \frac{nm}{n(m - 1) + m} v^{\alpha - 1}|\nabla v|^{m-2}v_i v_jE^i_{j} \Bigg],
		\end{aligned}
	\end{equation}
that is	
	\begin{align*}
		0 &= W +  {n\over n-1}v^{\alpha} E^i_{j}E^j_{i} + \Bigg[ {n\alpha\over n-1}- \frac{nm(p - P)}{n(m-1)+m} \Bigg]v^{\alpha - 1}|\nabla v|^{m-2}v_i v_jE^i_{j}\\
		&- \frac{nqU}{n(m - 1) + q} M v^{\alpha}|\nabla v|^{ q - 2} E^i_j v_i v_j  -(p - P)(\alpha-1)\frac{n(m-1)}{n(m-1)+m}v^{\alpha - 2}|\nabla v|^{2m}  \\
		&+ T v^{\alpha + 2p} + (U - 1)M^2 v^{\alpha}|\nabla v|^{2q} + \Bigg[ - (\alpha + p)T - P \Bigg]v^{\alpha + p-1}|\nabla v|^m\\
		&+ \Bigg[ \alpha + p - P - \frac{n(m - 1)\alpha U}{n(m - 1) + q} \Bigg]Mv^{\alpha - 1}|\nabla v|^{m + q } + \Big(T + U - 1\Big)Mv^{\alpha + p}|\nabla v|^ q.
	\end{align*}
	Define
	$$
		L_{ij} := v_{i} v_{j} - \frac{1}{n}|\nabla v|^2 \delta_{ij},
	$$
	then
	$$
		|L_{ij}|^2 = \frac{n - 1}{n}|\nabla v|^4,
	$$
	and
	\begin{align*}
		0 &\geq W - \frac{n - 1}{4n} v^{\alpha} \Bigg\{ \Bigg[ {n\alpha\over n-1}- \frac{nm(p - P)}{n(m-1)+m} \Bigg]v^{-1}|\nabla v|^{m-2}L_{ij}  - \frac{nqU}{n(m - 1) + q} M |\nabla v|^{ q - 2} L_{ij}  \Bigg\}^2 \\
		&-(p - P)(\alpha-1)\frac{n(m-1)}{n(m-1)+m}v^{\alpha - 2}|\nabla v|^{2m}  \\
		&+ T v^{\alpha + 2p} + (U - 1)M^2 v^{\alpha}|\nabla v|^{2q} + \Bigg[ - (\alpha + p)T - P \Bigg]v^{\alpha + p-1}|\nabla v|^m\\
		&+ \Bigg[ \alpha + p - P - \frac{n(m - 1)\alpha U}{n(m - 1) + q} \Bigg]Mv^{\alpha - 1}|\nabla v|^{m + q } + \Big(T + U - 1\Big)Mv^{\alpha + p}|\nabla v|^ q\\
		&= W - \frac{(n - 1)^2}{4n^2} v^{\alpha}|\nabla v|^4 \Bigg\{ \Bigg[ {n\alpha\over n-1}- \frac{nm(p - P)}{n(m-1)+m} \Bigg]v^{-1}|\nabla v|^{m-2}  - \frac{nqU}{n(m - 1) + q} M |\nabla v|^{ q - 2}   \Bigg\}^2 \\
		&-(p - P)(\alpha-1)\frac{n(m-1)}{n(m-1)+m}v^{\alpha - 2}|\nabla v|^{2m}  \\
		&+ T v^{\alpha + 2p} + (U - 1)M^2 v^{\alpha}|\nabla v|^{2q} + \Bigg[ - (\alpha + p)T - P \Bigg]v^{\alpha + p-1}|\nabla v|^m\\
		&+ \Bigg[ \alpha + p - P - \frac{n(m - 1)\alpha U}{n(m - 1) + q} \Bigg]Mv^{\alpha - 1}|\nabla v|^{m + q } + \Big(T + U - 1\Big)Mv^{\alpha + p}|\nabla v|^ q.\\
\end{align*}
After integration,
\begin{align}\label{last equation}
		RHS&= W + \Bigg\{ -(p - P)(\alpha-1)\frac{n(m-1)}{n(m-1)+m} - \frac{(n - 1)^2}{4n^2}\Bigg[ {n\alpha\over n-1}- \frac{nm(p - P)}{n(m-1)+m} \Bigg]^2 \Bigg\}v^{\alpha - 2}|\nabla v|^{2m} \nonumber\\
		&+ T v^{\alpha + 2p} + \Bigg\{ U - 1 - \frac{(n - 1)^2}{4n^2}\left[ \frac{nqU}{n(m - 1) + q}\right]^2 \Bigg\}M^2 v^{\alpha}|\nabla v|^{2q} \nonumber\\
		&+ \Bigg[ - (\alpha + p)T - P \Bigg]v^{\alpha + p-1}|\nabla v|^m  \nonumber\\
		&+ \Big\{ \alpha + p - P - \frac{n(m - 1)\alpha U}{n(m - 1) + q} + \frac{(n - 1)^2}{2n^2}\big[ {n\alpha\over n-1}- \frac{nm(p - P)}{n(m-1)+m} \big] \frac{nqU}{n(m - 1) + q}  \Big\}Mv^{\alpha - 1}|\nabla v|^{m + q } \nonumber\\
		&+ \Big(T + U - 1\Big)Mv^{\alpha + p}|\nabla v|^ q. \nonumber\\
	\end{align}
	If the following conditions hold at the same time:
	\begin{numcases}{}
		\label{sec2_condition1} -(p - P)(\alpha-1)\frac{n(m-1)}{n(m-1)+m} - \frac{(n - 1)^2}{4n^2}\Bigg[ {n\alpha\over n-1}- \frac{nm(p - P)}{n(m-1)+m} \Bigg]^2 > 0,\\
		\label{sec2_condition2} T > 0,\\	
		\label{sec2_condition3}  U - 1 - \frac{(n - 1)^2}{4n^2}\left[ \frac{nqU}{n(m - 1) + q}\right]^2 > 0,\\
		\label{sec2_condition4} - (\alpha + p)T - P\geq 0,\\
		\label{sec2_condition5} \Bigg\{ \alpha + p - P - \frac{n(m - 1)\alpha U}{n(m - 1) + q} + \frac{(n - 1)^2}{2n^2}\Bigg[ {n\alpha\over n-1}- \frac{nm(p - P)}{n(m-1)+m} \Bigg]\cdot \frac{nqU}{n(m - 1) + q}  \Bigg\}M \geq 0,\quad \quad\\
		\label{sec2_condition6} \Big(T + U - 1\Big)M \geq 0,
	\end{numcases}
	then there always exists $\varepsilon > 0$, such that
	\begin{equation}\label{sub4_eq1}
		\begin{aligned}
			&\varepsilon v^{\alpha - 2}|\nabla v|^{2m} + \varepsilon v^{\alpha}|\nabla v|^{2m - 6}\left(\sum_j v_j v_{ij}\right)^2 + \varepsilon v^{\alpha + 2p} + \varepsilon v^{\alpha}|\nabla v|^{2q} \\
		&\leq B_1 (v^{\alpha - 1} |\nabla v| ^{ 2m-2}v_i)_i + B_2(v^{\alpha }|\nabla v|^q X^i)_i + B_3(v^{\alpha + p} X^i)_i + B_4(v^{\alpha}X^i_{j} X^j )_i .
		\end{aligned}
	\end{equation}
	To make sure (\ref{sec2_condition1}) hold, we can choose
	\begin{align*}
		\alpha = \frac{-mn - m + 2n}{n(m - 1) + m}(p - P),
	\end{align*}
	and
	\begin{align*}
		m - 1 < p - P < \frac{(m - 1)n + m}{n - m}.
	\end{align*}

	\section{Young's inequality and cut-off functions}
	In this part, we assume \eqref{sub4_eq1} holds and prove that $|\nabla v| \equiv 0$. Define $\eta$ is a smooth cut-off function, satisfying that
	\begin{equation*}
		\begin{aligned}
			&\eta \equiv 1 \quad \text{in} \quad B_{R},\\
			&\eta \equiv 0 \quad \text{in} \quad \mathbb R^n\backslash B_{2R}.\\
		\end{aligned}
	\end{equation*}
	 Multiply \eqref{sub4_eq1} by $\eta^{\delta}$ and integrate over $\mathbb R^n$,
	\begin{align*}
		&\varepsilon\int_{R^n}  v^{\alpha - 2}|\nabla v|^{2m} \eta^{\delta} + \varepsilon \int_{R^n} v^{\alpha}|\nabla v|^{2m - 6}\left(\sum_j v_j v_{ij}\right)^2\eta^{\delta} + \varepsilon \int_{R^n} v^{\alpha + 2p}\eta^{\delta} + \varepsilon \int_{R^n} v^{\alpha}|\nabla v|^{2q}\eta^{\delta}\\
		&\leq B_1\delta \int_{B_{2R}\backslash B_R}  v^{\alpha - 1} |\nabla v| ^{ 2m-2}v_i \eta^{\delta - 1}\eta_i + B_2 \delta \int_{B_{2R}\backslash B_R} v^{\alpha }|\nabla v|^{m + q - 2} v_i \eta^{\delta - 1}\eta_i \\
&+ B_3 \delta \int_{B_{2R}\backslash B_R} v^{\alpha + p} |\nabla v|^{m - 2}v_i \eta^{\delta - 1}\eta_i 
		+ B_4 \delta  \int_{B_{2R}\backslash B_R} v^{\alpha}X^i_{j} X^j\eta^{\delta - 1}\eta_i.
	\end{align*}
	Since we know that
	\begin{align*}
		& \int_{B_{2R}\backslash B_R} v^{\alpha}X^i_{j} X^j\eta^{\delta - 1}\eta_i \\
		&= -\int_{B_{2R}\backslash B_R}  (v^{\alpha} X^j\eta^{\delta - 1}\eta_i)_jX^i\\
		&= -\alpha \int_{B_{2R}\backslash B_R} v^{\alpha - 1}|\nabla v|^{2m - 2}v_i \eta^{\delta - 1}\eta_i - \int_{B_{2R}\backslash B_R} v^{\alpha} \Delta_m v \eta^{\delta - 1}\eta_i X^i - \int_{B_{2R}\backslash B_R}  v^{\alpha} X^j(\eta^{\delta - 1}\eta_i)_jX^i\\
		&= -\alpha \int_{B_{2R}\backslash B_R} v^{\alpha - 1}|\nabla v|^{2m - 2}v_i \eta^{\delta - 1}\eta_i + \int_{B_{2R}\backslash B_R} v^{\alpha} ( v^p + M|\nabla v|^q) \eta^{\delta - 1}\eta_i X^i\\
& - \int_{B_{2R}\backslash B_R}  v^{\alpha} |\nabla v|^{2m - 4}v_i v_j (\eta^{\delta - 1}\eta_i)_j\\
        &=-\alpha \int_{B_{2R}\backslash B_R} v^{\alpha - 1}|\nabla v|^{2m - 2}v_i \eta^{\delta - 1}\eta_i +\int_{B_{2R}\backslash B_R} v^{\alpha + p} |\nabla v|^{m - 2}v_i \eta^{\delta - 1}\eta_i+M\int_{B_{2R}\backslash B_R} v^{\alpha }|\nabla v|^{m + q - 2} v_i \eta^{\delta - 1}\eta_i\\
        &-\int_{B_{2R}\backslash B_R} v^{\alpha} |\nabla v|^{2m - 4}v_i v_j (\eta^{\delta - 1}\eta_i)_j,
	\end{align*}
	then we get
	\begin{align*}
		&\varepsilon\int_{R^n}  v^{\alpha - 2}|\nabla v|^{2m} \eta^{\delta} + \varepsilon \int_{R^n} v^{\alpha}|\nabla v|^{2m - 6}\left(\sum_j v_j v_{ij}\right)^2\eta^{\delta} + \varepsilon \int_{R^n} v^{\alpha + 2p} \eta^{\delta}+ \varepsilon \int_{R^n} v^{\alpha}|\nabla v|^{2q} \eta^{\delta}\\
		&\leq \frac{\varepsilon}{2} \int_{B_{2R}\backslash B_R}  v^{\alpha - 2}|\nabla v|^{2m} \eta^{\delta} + \frac{C}{R^2} \int_{B_{2R}\backslash B_R} v^{\alpha} |\nabla v|^{2m - 2} \eta^{\delta - 2}\\
		&+ \frac{\varepsilon}{2} \int_{B_{2R}\backslash B_R}  v^{\alpha}|\nabla v|^{2q} \eta^{\delta} +  \frac{C}{R^2} \int_{B_{2R}\backslash B_R} v^{\alpha} |\nabla v|^{2m - 2} \eta^{\delta - 2}\\
		&+ \frac{\varepsilon}{2} \int_{B_{2R}\backslash B_R}  v^{\alpha + 2p} \eta^{\delta} +  \frac{C}{R^2} \int_{B_{2R}\backslash B_R} v^{\alpha} |\nabla v|^{2m - 2} \eta^{\delta - 2}.\\
	\end{align*}
Therefore, we get that

	\begin{align}\label{Y1}
     &\int_{R^n}  v^{\alpha - 2}|\nabla v|^{2m} \eta^{\delta} +  \int_{R^n} v^{\alpha + 2p}\eta^{\delta} \leq \frac{C}{R^2} \int_{B_{2R}\backslash B_R} v^{\alpha} |\nabla v|^{2m - 2} \eta^{\delta - 2} .
	\end{align}
	Define $p_1, q_1, \sigma_1 > 0$, such that
	\begin{equation}\label{parameter3}
		\frac{1}{p_1} + \frac{1}{q_1} + \frac{1}{\sigma_1} = 1, \text{ and } p_1, q_1, \sigma_1 > 0.
	\end{equation}
So by Young's inequality, we know
	\begin{equation}\label{Y2}
		\begin{aligned}
			&R^{-2} \int_{B_{2R}\backslash B_R} v^{\alpha}|\nabla v|^{2m - 2}\eta^{\delta - 2}\\
			&= R^{-2}\int_{B_{2R}\backslash B_R} v^{\alpha - A}|\nabla v|^{2m - 2 }\cdot v^{A}\eta^{\delta - 2}\\
			&\leq \varepsilon_0 \int_{B_{2R}\backslash B_R} v^{\alpha - 2}|\nabla v|^{2m}\eta^\delta + \varepsilon_0 \int_{B_{2R}\backslash B_R} v^{\alpha + 2p}\eta^{\delta} + CR^{-2\sigma_1} \int_{B_{2R}\backslash B_R}\eta^{\delta - 2\sigma_1}.
\end{aligned}
\end{equation}		
Combine (\ref{Y1}) and (\ref{Y2}), it follows that

	\begin{equation}
		\begin{aligned}	 & \int_{R^n} v^{\alpha +  2p}\eta^{\delta} + \int_{R^n} v^{\alpha - 2}|\nabla v|^{2m}\eta^\delta \\
			&\leq CR^{-2\sigma_1} \int_{B_{2R}\backslash B_R} \eta^{\delta - 2\sigma_1}\\
			&\leq CR^{n - 2\sigma_1} \rightarrow 0, \text{as R tends to infinity}.
		\end{aligned}
	\end{equation}
The last two steps come from the condition that
	\begin{equation}\label{parameter5}
		\begin{aligned}
			\delta - 2\sigma_1 >0,\\
			n - 2\sigma_1  < 0.
		\end{aligned}
	\end{equation}
	So we always take $\delta$ large enough.
	In conclusion, we need
	\begin{equation}\label{young_con}
		\begin{aligned}
			&\frac{\alpha + 2p}{A} = p_1 > 0,\\
			&\frac{\alpha - 2}{\alpha - A} = \frac{2m}{2m - 2} = q_1 > 0,\\
			&1 - \frac{2}{n} < \frac{1}{p_1} + \frac{1}{q_1} < 1.
		\end{aligned}
	\end{equation}
	
	Therefore, as long as the conditions \eqref{sec2_condition1} - \eqref{sec2_condition6} are satisfied, and there exists such $p_1, q_1, \sigma_1$, then $v \equiv 0$.
	
	By \eqref{young_con}, we know
	\begin{equation}\label{p1}
		\begin{aligned}
			&A = \frac{\alpha + 2m - 2}{m},\\
			&\frac{1}{p_1} + \frac{1}{q_1} = \frac{m - 1}{m} + \frac{\alpha + 2m - 2}{m(\alpha + 2p)}.
		\end{aligned}
	\end{equation}
Next, we set out to prove
$$
1 - \frac{2}{n} < \frac{1}{p_1} + \frac{1}{q_1} < 1.
$$

Firstly, for the left hand:
	\begin{align*}
		&1 - \frac{2}{n} < \frac{1}{p_1} + \frac{1}{q_1},\\
		\Leftrightarrow ~& \frac{2}{n} - \frac{1}{m} + \frac{\alpha + 2m - 2}{m(\alpha + 2p)} > 0,\\
		\Leftrightarrow ~& \alpha > -2p - n + \frac{n(p + 1)}{m}.
	\end{align*}
	As we choose in the previous section,
	\begin{align*}
		\alpha = \frac{-mn - m + 2n}{n(m - 1) + m}(p - P),
	\end{align*}
	then it means to show
	\begin{align*}
		& \frac{-mn - m + 2n}{n(m - 1) + m}(p - P) > -2p - n + \frac{n(p + 1)}{m}.
	\end{align*}
	Since $m - 1 < p - P < \frac{(m - 1)n + m}{n - m}$, we only need to show that
	\begin{numcases}{}
		\frac{-mn - m + 2n}{n(m - 1) + m}(m - 1 - P) > -2(m - 1),\\
		\frac{-mn - m + 2n}{n(m - 1) + m}\cdot \frac{(m - 1)n + m}{n - m} - \frac{-mn - m + 2n}{n(m - 1) + m}P > \frac{n - 2m}{m}\frac{(m - 1)n + m}{n - m} - n + \frac{n}{m}.\quad \quad
	\end{numcases}
	It is equivalent to be
	\begin{numcases}{}
		\label{P_1} {(m-1)}(mn + m) - (-mn - m + 2n)P > 0,\\	
		\label{P_2}\frac{m^2}{m(n - m)} - \frac{-mn - m + 2n}{n(m - 1) + m}P > 0.
	\end{numcases}

Next, for the right hand:
\\
by (\ref{p1}), it's namely to prove
$$
 \frac{\alpha + 2m - 2 - \alpha - 2p}{m(\alpha + 2p)} < 0.
$$
We only need to show that
\begin{align*}
	\alpha + 2m - 2 > 0,
\end{align*}
then
\begin{align*}
	&\alpha + 2p > \alpha + 2m - 2 > 0,\\
	&p_1 > 0.
\end{align*}
There are three different cases:

	\begin{itemize}
		\item If $\alpha + 2m - 2 > 0$, then we can get desired $p_1, q_1 > 0$.
		\item If $\alpha + 2m - 2 \leq 0$ and $M > 0$, then we have the following Lemma:

	\begin{lemma}
		If we choose the same $\alpha$ as before, then
		\begin{equation}
			\begin{aligned}
				\int_{R^n} v^{\alpha + 2p}\eta^{\delta} \leq CR^{-2m} \int_{B_{2R}\backslash B_R} v^{\alpha + 2m - 2} \eta^{\delta - 2m}.
			\end{aligned}
		\end{equation}
		
	\end{lemma}	
		\begin{proof}
			By \eqref{Y1} and Young's inequality , we have
		\begin{align*}
			&\int_{R^n}  v^{\alpha - 2}|\nabla v|^{2m} \eta^{\delta} +  \int_{R^n} v^{\alpha + 2p}\eta^{\delta}\\ 
            &\leq \frac{C}{R^2} \int_{B_{2R}\backslash B_R} v^{\alpha} |\nabla v|^{2m - 2} \eta^{\delta - 2} \\
			&\leq CR^{-2m} \int_{B_{2R}\backslash B_R} v^{\alpha + 2m - 2} \eta^{\delta - 2m} + \frac{1}{2} \int_{B_{2R}\backslash B_R}  v^{\alpha - 2}|\nabla v|^{2m} \eta^{\delta}
		\end{align*}
		with $(m, \frac{m}{m - 1})$.
		So we get that
		\begin{equation*}
			\begin{aligned}
				\int_{R^n} v^{\alpha + 2p}\eta^{\delta} \leq CR^{-2m} \int_{B_{2R}\backslash B_R} v^{\alpha + 2m - 2} \eta^{\delta - 2m}.
			\end{aligned}
		\end{equation*}
		\end{proof}
	
	Besides, we need the following result which comes from Lemma 2.3 in \cite{Serrin1} by Serrin-Zou.
	\begin{lemma}\label{lemma_SZ}
		Suppose $\{ |x| > R > 0\} \subset \Omega $ and $1 < m < n$ . Let $v$ be a positive solution of the inequality
		\begin{align*}
			\Delta_m v \leq 0, x\in \Omega.
		\end{align*}
		Then there exists a constant $C = C(m, n, \min\limits_{|x| = 2R} v, R)  $ such that
		\begin{align*}
			v(x) \geq C|x|^{-\frac{n - m}{m - 1}}, \text{when } |x| > 2R.
		\end{align*}

	\end{lemma}	
	Therefore, if $M > 0$, we know that $\Delta_m v \leq 0$, then
	\begin{align*}
		v(x) \geq CR^{-\frac{n - m}{m - 1}}.
	\end{align*}
	Thus we get that
	\begin{align*}
		\int_{R^n} v^{\alpha + 2p}\eta^{\delta} &\leq CR^{-2m} \int_{B_{2R}\backslash B_R}v^{\alpha + 2m - 2} \eta^{\delta - 2m}\\
		&\leq C R^{-2m + n - (\alpha + 2m - 2)\cdot \frac{n - m}{m - 1}}.
	\end{align*}
	We hope that
	\begin{align*}
		-2m + n - (\alpha + 2m - 2)\cdot \frac{n - m}{m - 1} < 0.
	\end{align*}
	
	If $-mn - m + 2n \geq 0$, then $\alpha \geq 0$ and $\alpha + 2m - 2 > 0$. This is the first case. So we only need to consider the case that $-mn - m + 2n < 0$:

Since
	\begin{align*}
		&\alpha + 2m - 2\\
		= & \frac{-mn - m + 2n}{n(m - 1) + m}p - \frac{-mn - m + 2n}{n(m - 1) + m} P + 2m - 2\\
		\geq &\frac{-mn - m + 2n}{n - m} - \frac{-mn - m + 2n}{n(m - 1) + m} P + 2m - 2\\
		=&  \frac{-mn - m + 2n + 2mn - 2m^2 - 2n + 2m}{n - m} - \frac{-mn - m + 2n}{n(m - 1) + m} P\\
		=&  \frac{m(n + 1 - 2m)}{n - m} - \frac{-mn - m + 2n}{n(m - 1) + m}P,\\
	\end{align*}
then
	
	\begin{align}\label{P_3}
		&-2m + n - (\alpha + 2m - 2)\cdot \frac{n - m}{m - 1}\nonumber\\
		\leq&  -2m + n - \frac{m(-2m + n + 1)}{n - m}\cdot \frac{n - m}{m - 1} +  \frac{-mn - m + 2n}{n(m - 1) + m}P\cdot  \frac{n - m}{m - 1}\nonumber \\
		=&  -2m + n - \frac{m(-2m + n + 1)}{m - 1} +  \frac{-mn - m + 2n}{n(m - 1) + m}P\cdot  \frac{n - m}{m - 1} \nonumber\\
		=& \frac{-2m^2 + 2m + mn - n + 2m^2 - mn - m}{m - 1} +  \frac{-mn - m + 2n}{n(m - 1) + m}P\cdot  \frac{n - m}{m - 1} \nonumber\\
		=&\frac{m - n}{m - 1} +  \frac{-mn - m + 2n}{n(m - 1) + m}P\cdot  \frac{n - m}{m - 1} \nonumber\\
		=& \frac{m - n}{m - 1}\left[ 1 - \frac{-mn - m + 2n}{n(m - 1) + m}P \right].
	\end{align}
	
	In conclusion,  it follows by (\ref{P_1}), (\ref{P_2}) and (\ref{P_3}) that if 
$$M > 0, 1 < m < n, m - 1 < p - P < \frac{(m - 1)n + m}{n - m}, -P_m < P < P_m,$$ where
	\begin{align*}
		P_m = \min \Big\{ \frac{(mn + m)(m-1)}{|-mn - m + 2n|}, \frac{m^2}{m(n - m)}\left[ \frac{|-mn - m + 2n|}{n(m - 1) + m}\right]^{-1}, \frac{n(m - 1) + m}{|-mn - m + 2n|} \Big \},
	\end{align*}
	 then we can deduce by \eqref{sub4_eq1} that
	\begin{align*}
		v \equiv 0.
	\end{align*}
	\item If $\alpha + 2m - 2 \leq 0$ and $M < 0$. The condition $\Delta_m v \leq 0$ does not hold any more and we can not use Lemma \ref{lemma_SZ} directly. But inspired by \cite{F-Y-Y}, if we define $v = w^{\sigma}$ with $\sigma > 1$, we can get that $\Delta_m w \leq 0$ under some conditions for $p, P, m, M$.
	\begin{align*}
		-\Delta_m w &= (\sigma - 1)(m - 1)w^{-1}|\nabla w|^m + \sigma^{1 - m}w^{m + \sigma(p - m + 1) - 1} \\
		&+ M\sigma^{q - m + 1}w^{(\sigma - 1)(q - m + 1)}|\nabla w|^q,
	\end{align*}
	and then setting $z = |\nabla v|^m$ yields
	\begin{align*}
		-\Delta_m w = \sigma^{1 - m}w^{-1}\Psi(z),
	\end{align*}
	where
	\begin{align*}
		\Psi(z) = \sigma^{m - 1} (\sigma - 1)(m - 1)z + M\sigma^{q}w^{(\sigma - 1)(q - m + 1) + 1}z^{\frac{q}{m}} + w^{m + \sigma(p - m + 1) } .
	\end{align*}
	Since $q = \frac{mp}{p + 1}$, it is easy to see that $\Psi(z)$ achieves its minimum at
	\begin{align*}
		z_0 = \left[ \frac{|M| ~p~\sigma^{1 - \frac{m}{p + 1}}}{(\sigma - 1)(m - 1)(p + 1)} \right]^{p + 1}v^{m + \sigma(p - m + 1)},
	\end{align*}
	and
	\begin{align*}
		\Psi(z_0) = \left[1 - \left(\frac{|M|}{p + 1}\right)^{p + 1} \frac{(\sigma p)^p}{(\sigma - 1)^p(m - 1)^p} \right]v^{m + \sigma (p - m + 1)}.
	\end{align*}
	If
	\begin{equation}
		1 - \left(\frac{|M|}{p + 1}\right)^{p + 1} \frac{(\sigma p)^p}{(\sigma - 1)^p(m - 1)^p} \geq 0,
	\end{equation}
	then we know that $\Delta_m w \geq 0$ and by Lemma \ref{lemma_SZ} we get that
	\begin{align*}
		v = w^{\sigma} \geq CR^{-\frac{n - m}{m - 1}\sigma}.
	\end{align*}
	Thus we hope the follow conditions will hold at the same time:
	\begin{numcases}{}
		\label{negative1} -2m + n - (\alpha + 2m - 2)\cdot \frac{n - m}{m - 1}\sigma < 0,\\
		\label{negative2}1 - \left(\frac{|M|}{p + 1}\right)^{p + 1} \cdot\frac{(\sigma p)^p}{(\sigma - 1)^p(m - 1)^p} \geq 0.
	\end{numcases}
	
	Observe that
	\begin{align*}
		&-2m + n - (\alpha + 2m - 2)\cdot \frac{n - m}{m - 1}\sigma\\
		=&\frac{m - n}{m - 1}\left[1 + \frac{m(-2m + n + 1)}{n - m}(\sigma - 1) \right] - \frac{m - n}{m - 1}\cdot\frac{-mn - m + 2n}{n(m - 1) + m}P \sigma.
	\end{align*}
	If $-2m + n + 1 \geq 0$, we choose $\sigma = 2$; if $-2m + n + 1 < 0$, we choose
	\begin{align*}
		\frac{m(-2m + n + 1)}{n - m}(\sigma - 1) = -\frac{1}{2}.
	\end{align*}
	Then there exists $M_3 = M_3(n, m, p) > 0$, such that if $-M_3 < M < 0$, then the condition \eqref{negative2} hold.
	Besides, for condition \eqref{negative1} we have
	\begin{align}\label{P_4}
		& -2m + n - (\alpha + 2m - 2)\cdot \frac{n - m}{m - 1}\sigma \nonumber\\
		\leq &\frac{m - n}{m - 1}\cdot \frac{1}{2} - \frac{m - n}{m - 1}\cdot\frac{-mn - m + 2n}{n(m - 1) + m}P \sigma \nonumber\\
		=& \frac{m - n}{m - 1}\cdot \left[  \frac{1}{2} - \frac{-mn - m + 2n}{n(m - 1) + m}P \sigma\right]. \nonumber\\
	\end{align}

	In conclusion, combined with  (\ref{P_1}), (\ref{P_2}) and (\ref{P_4}), if 
$$-M_3 < M < 0, 1 < m < n, m - 1 < p - P < \frac{(m - 1)n + m}{n - m}, -P_m < P < P_m,$$ where
	\begin{align*}
		P_m = \min \Big\{ \frac{(mn + m)(m-1)}{|-mn - m + 2n|}, \frac{m^2}{m(n - m)}\left[ \frac{|-mn - m + 2n|}{n(m - 1) + m}\right]^{-1},  \frac{n(m - 1) + m}{2\sigma |-mn - m + 2n|} \Big \},
	\end{align*}
	 then we can deduce by \eqref{sub4_eq1} that
	\begin{align*}
		v \equiv 0.
	\end{align*}
	
	\end{itemize}

	\section{Case $M > 0$ and $M < 0$}
In this section, we are working on proving \eqref{sec2_condition1} - \eqref{sec2_condition6} hold, and talk about them separately according to the sign of $M$.
\\
{\bf Case 1: $M > 0$}

	\begin{theorem}\label{thm1}
		Let $n \geq 2$, $m - 1 < p \leq \frac{(m - 1)n + m}{n - m}$, $q = \frac{mp}{p + 1}$, then there exists $M_1 > 0$, which depends on $n, m, p$, such that for any $0 < M < M_1$, all the nonnegative solutions of \eqref{main equation} are $v \equiv 0$.
	\end{theorem}
	
	\begin{proof}[Proof of Theorem \ref{thm1}]
		In this proof, we choose $P = T = 0$ at first and then
		\begin{align*}
			\alpha &= \frac{-mn - m + 2n}{n(m - 1) + m}p,\\
			\alpha + p &= \frac{np}{n(m - 1) + m} > 0,
		\end{align*}
		then the condition \eqref{sec2_condition1} holds for $m - 1 < p < \frac{(m - 1)n + m}{n - m}$. And the condition \eqref{sec2_condition5} becomes that:
		\begin{align*}
			0 &\leq \alpha + p  - \frac{n(m - 1)\alpha U}{n(m - 1) + q} + \frac{(n - 1)^2}{2n^2}\Bigg[ {n\alpha\over n-1}- \frac{nmp}{n(m-1)+m} \Bigg]\cdot \frac{nqU}{n(m - 1) + q}\\
			&= \frac{np}{n(m - 1) + m}  - \frac{-mn - m + 2n}{n(m - 1) + m}p\cdot \frac{n(m - 1) U}{n(m - 1) + q} \\
			&+ \frac{(n - 1)^2}{2n^2}\Bigg[ {n\over n-1}\cdot \frac{-mn - m + 2n}{n(m - 1) + m}p - \frac{nmp}{n(m-1)+m} \Bigg]\cdot \frac{nqU}{n(m - 1) + q}\\
			&= \frac{np}{n(m - 1) + m} + \frac{np U}{(nm - n + q)(nm - n + m)} \Bigg[ (mn + m - 2n)(m - 1) \\
			&+ \frac{(n - 1)^2q}{2n} \left( \frac{-mn - m + 2n}{n - 1} - m \right) \Bigg]\\
			&= \frac{np}{nm - n + m} + \frac{np(m - 1) U}{(nm - n + q)(nm - n + m)} \Bigg[ mn + m - 2n - (n - 1)q \Bigg]\\
            &=\frac{np}{(nm - n + q)(nm - n + m)}\Big[nm - n + q +(m - 1)U\big( mn + m - 2n - (n - 1)q  \big)\Big].
		\end{align*}
		Define
		\begin{align*}
			f_1(m, q) := nm - n + q + (m - 1) \left(1 + \frac{1}{mn}\right) \Bigg[ mn + m - 2n - (n - 1)q \Bigg].
		\end{align*}
		\begin{claim}\label{claim_1}
			For any $m - 1 < q \leq m - 1 + \frac{m}{n}$, $1 < m < n$, we have
			\begin{align*}
				f_1(m, q) > 0.
			\end{align*}
		\end{claim}
\begin{proof}[Proof of claim \ref{claim_1}]
			We only need to show that
			\begin{align*}
				f_1(m, m - 1)\quad \mathrm{and}\quad f_1\left(m, m - 1 + \frac{m}{n}\right) > 0.
			\end{align*}
			In fact, we have
			\begin{align*}
				f_1(m, m - 1) &= nm - n + m - 1 + \left(1 + \frac{1}{mn}\right)(m - 1) \Bigg[ mn + m - 2n - (n - 1)(m - 1) \Bigg]\\
				&= (m - 1)\left(1 + \frac{1}{mn}\right)\Bigg[ n + 1 - \frac{n + 1}{mn + 1} + mn + m - 2n - (n - 1)(m - 1) \Bigg]\\
				&= (m - 1)\left(1 + \frac{1}{mn}\right)\left( 2m - \frac{n + 1}{mn + 1}\right)  > 0,
			\end{align*}
			and
			\begin{align*}
				&f_1\left(m, m - 1 + \frac{m}{n}\right)\\
				&= nm - n + m - 1 + \frac{m}{n} + (m - 1) \left(1 + \frac{1}{mn}\right)\Bigg[ mn + m - 2n - (n - 1)\left(m - 1 + \frac{m}{n}\right) \Bigg]\\
				&= \frac{m^2n^3 - mn^3 + m^2n^2 - mn^2 + m^2n + (m - 1)(mn + 1)\Big[ mn^2 + mn - 2n^2 - (n - 1)\left(mn - n + m\right) \Big]}{mn^2}\\
				&= \frac{m^2n^3 - mn^3 + m^2n^2 - mn^2 + m^2n + (m^2n + m - mn - 1) \Big[ mn - n^2  - n + m\Big]}{mn^2}>0.
			\end{align*}
			Define
			\begin{align*}
				f(n)&:=m^2n^3 - mn^3 + m^2n^2 - mn^2 + m^2n + (m^2n + m - mn - 1) \Big[ mn - n^2  - n + m\Big]\\
				&= m^2n^3 - mn^3 + m^2n^2 - mn^2 + m^2n + m^3n^2 - m^2n^3 - m^2n^2 + m^3n + m^2n - mn^2 - mn + m^2\\
				&- m^2n^2 + mn^3 + mn^2 - m^2n - mn + n^2 + n - m\\
				&=  m^3n^2  + m^3n + m^2n - mn^2 - 2mn + m^2 - m^2n^2 + n^2 + n - m\\
				&= (m^3 - m^2 - m + 1)n^2 + (m^3 + m^2 - 2m + 1)n + m^2 - m ,
			\end{align*}
		it can be seen that 
	\begin{align*}
				f'(n)&=2n (m^3 - m^2 - m + 1)+ m^3+ m^2  -2 m+1\\
				&> 2m (m^3 - m^2 - m + 1)+ m^3+ m^2  -2 m+1\\
				&= 2m^4  - m^3-m^2+1\\
				&=  m^2 ( m - 1)( 2m + 1)+1>0.
			\end{align*}
		so $f(n)>f(m)=m^5>0	.$

		\end{proof}
		Since $m - 1 < p \leq \frac{(m - 1)n + m}{n - m}$, and $q={mp\over p+1}$, then we obtain the range of $q$, that is
$$
m - 1 < q \leq m - 1 + \frac{m}{n}.
$$ Therefore, if we choose $U = 1 + \frac{1}{mn}$, then the condition \eqref{sec2_condition5} and \eqref{sec2_condition6} hold. That is if we choose
		\begin{align*}
			P &= T =  0,\\
			\alpha &= \frac{-mn - m + 2n}{n(m - 1) + m}(p - P),\\
			U &= 1 + \frac{1}{mn},
		\end{align*}
		then for any $m - 1 < p < \frac{(m - 1)n + m}{n - m}$, we have
		\begin{align*}
			\alpha + p  - \frac{n(m - 1)\alpha U}{n(m - 1) + q} + \frac{(n - 1)^2}{2n^2}\Bigg[ {n\alpha\over n-1}- \frac{nmp}{n(m-1)+m} \Bigg]\cdot \frac{nqU}{n(m - 1) + q} > 0,
		\end{align*}
		 and the conditions \eqref{sec2_condition1} and \eqref{sec2_condition6} hold.
It remains to show the condition  \eqref{sec2_condition3} hold, while if
\begin{align*}
  U - 1 - \frac{(n - 1)^2}{4n^2}\left[ \frac{nqU}{n(m - 1) + q}\right]^2< 0,
\end{align*}
	 then by Young's inequality, it follows that
		\begin{equation}\label{Young1}
			\begin{aligned}
				M^2 v^{\alpha}|\nabla v|^{2q} \leq K|M| v^{\alpha - 1}|\nabla v|^{m + q} + \frac{\Big(K\cdot \frac{p + 1}{p}\Big)^{-p}}{p + 1} |M|^{p + 2} v^{\alpha + p}|\nabla v|^q,
			\end{aligned}
		\end{equation}
		with $(\frac{p + 1}{p}, p + 1)$, by (\ref{last equation}) we only need 
	\begin{numcases}{}
		\label{sec2_condition} \alpha + p  - \frac{n(m - 1)\alpha U}{n(m - 1) + q} + \frac{(n - 1)^2}{2n^2}\Big[ {n\alpha\over n-1}- \frac{nmp}{n(m-1)+m} \Big]\cdot \frac{nqU}{n(m - 1) + q} \\
+K\Big(U - 1 - \frac{(n - 1)^2}{4n^2}\left[ \frac{nqU}{n(m - 1) + q}\right]^2\Big) =0,\\
		\label{sec2_condition}  U-1+\frac{(K{p+1\over p})^{-p}}{p+1}|M|^{p + 1} \Big( U - 1 - \frac{(n - 1)^2}{4n^2}\left[ \frac{nqU}{n(m - 1) + q}\right]^2 \Big)>0,
	\end{numcases}
then
\begin{align*}
&U - 1 - \frac{p^p}{(p+1)^{p+1}}|M|^{p + 1} \Big(- U + 1 + \frac{(n - 1)^2}{4n^2}\left[ \frac{nqU}{n(m - 1) + q}\right]^2 \Big)^{p+1}\\
&\cdot \Bigg[\alpha + p  - \frac{n(m - 1)\alpha U}{n(m - 1) + q} + \frac{(n - 1)^2}{2n^2}\Bigg[ {n\alpha\over n-1}- \frac{nmp}{n(m-1)+m} \Bigg]\cdot \frac{nqU}{n(m - 1) + q}\Bigg]^{-p}>0.
\end{align*}
Since $U$ has been chosen as before, by the inequality above, we can get $M<M_1$. Thus we can say there exists $M_1>0$, such that for any $0<M<M_1$, all the inequality conditions hold. Therefore, by continuity, when we choose $$-\varepsilon < P < 0, ~(\alpha + p)T + P = 0,$$ such that
\begin{align*}
	p - P <\frac{(m - 1)n + m}{n - m},
\end{align*}
then conditions \eqref{sec2_condition1} - \eqref{sec2_condition6} all hold.
Finally we explain the critical case $p = \frac{(m-1)n + m}{n - m}$. For fixed $M > 0$, we can still choose $0 < P < \varepsilon$, where $\varepsilon > 0$ is much smaller than $M$. Thus using the following Young's inequality, we can get the same result as $p < \frac{(m-1)n + m}{n - m}$:
	
		\begin{equation}\label{Young2}
			\begin{aligned}
				v^{\alpha + p - 1}|\nabla v|^m \leq \frac{q}{m}\left(J|M| \cdot \frac{m}{m - q}\right)^{-\frac{m - q}{q}} v^{\alpha + p}|\nabla v|^q + J|M| v^{\alpha - 1}|\nabla v|^{m + q}
			\end{aligned}
		\end{equation}
		with $(\frac{m}{q}, \frac{m}{m - q})$.

%By \eqref{Young1} and \eqref{Young2},  there exists $M_1 > 0$, such that for any $0 < M < M_1$, the conditions  \eqref{sec2_condition3} and \eqref{sec2_condition4} hold.
	\end{proof}

{\bf Case 2: $M <0$}
\\	
	\begin{theorem}\label{thm2}
		Let $n \geq 2$, $m - 1 < p < \frac{(m - 1)n + m}{n - m}$, $q = \frac{mp}{p + 1}$, then there exists $M_2 > 0$, which depends on $n, m, p$, such that for any $-M_2 < M < 0$, all the nonnegative solutions of \eqref{main equation} are $v \equiv 0$.
	\end{theorem}
	\begin{proof}[Proof of Theorem \ref{thm2}]
		We still choose
		\begin{align*}
			\alpha = \frac{-mn - m + 2n}{n(m - 1) + m}(p - P).
		\end{align*}
		There are two different cases:
		\begin{itemize}
			\item $1 < m \leq \frac{n}{2} + \frac{1}{4}$:
\\
			In this case, we know
			\begin{align*}
				& \alpha + p - P  - \frac{n(m - 1)\alpha U}{n(m - 1) + q} + \frac{(n - 1)^2}{2n^2}\Bigg[ {n\alpha\over n-1}- \frac{nm(p - P)}{n(m-1)+m} \Bigg]\cdot \frac{nqU}{n(m - 1) + q}\\
			&= \frac{n(p - P)}{nm - n + m} + \frac{n(p - P)(m - 1) U}{(nm - n + q)(nm - n + m)} \Bigg[ mn + m - 2n - (n - 1)q \Bigg],\\
			\end{align*}
			and
			\begin{align*}
				&mn + m - 2n - (n - 1)q\\
				&< mn + m - 2n - (n - 1)(m - 1)\\
				&= 2m - n - 1\leq - \frac{1}{2} .
			\end{align*}
			So we can choose $U > 0$, such that
			\begin{align*}
				\alpha + p  - \frac{n(m - 1)\alpha U}{n(m - 1) + q} + \frac{(n - 1)^2}{2n^2}\Bigg[ {n\alpha\over n-1}- \frac{nmp}{n(m-1)+m} \Bigg]\cdot \frac{nqU}{n(m - 1) + q} < 0.
			\end{align*}
			Besides, for any $m - 1 < p - P < \frac{ nm - n + m}{n - m}$, we have
			\begin{align*}
				&-(p - P)(\alpha-1)\frac{n(m-1)}{n(m-1)+m} - \frac{(n - 1)^2}{4n^2}\Bigg[ {n\alpha\over n-1}- \frac{nm(p - P)}{n(m-1)+m} \Bigg]^2\\
				&= \frac{p - P}{(nm - n + m)^2}(nm - n)\Bigg[ nm - n + m - (n - m)(p - P) \Bigg] > 0.
			\end{align*}
			So for any $m - 1 < p < \frac{ nm - n + m}{n - m}$, there always exists $P_0$, such that $-P_m < P_0 < 0$ and
			\begin{align*}
				&p - P_0 < \frac{ nm - n + m}{n - m},\\
				&\alpha + p > 0.
			\end{align*}
			Choose
			\begin{align*}
				T = \frac{-0.5P_0}{\alpha + p} > 0,
			\end{align*}
			then the conditions \eqref{sec2_condition1}, \eqref{sec2_condition2}, \eqref{sec2_condition4} and \eqref{sec2_condition5} all hold.
It remains to show \eqref{sec2_condition3} and \eqref{sec2_condition6}. If they don't hold, we can use Young's inequality,
		%\begin{equation}\label{Young3}
%			\begin{aligned}
%				|M|v^{\alpha - 1}|\nabla v|^{m + q} \leq J v^{\alpha - 2}|\nabla v|^{2m} + \frac{1}{p + 2} \left(J \cdot \frac{p + 2}{p + 1}\right)^{- p - 1} |M|^{p + 2} v^{\alpha + p}|\nabla v|^q,
%			\end{aligned}
%		\end{equation}
%		with $(\frac{p + 2}{p + 1} , p + 2)$, and
		\begin{equation}\label{Young4}
			\begin{aligned}
				|M| v^{\alpha + p}|\nabla v|^q \leq K v^{\alpha + 2p} + \frac{q}{m}\left(K\cdot \frac{m}{m - q}\right)^{- \frac{m - q}{q}}  |M|^{\frac{m}{q}} v^{\alpha + p - 1}|\nabla v|^m,
			\end{aligned}
		\end{equation}
		with $(\frac{m}{m - q}, \frac{m}{q})$.

Using \eqref{Young1} and \eqref{Young4}, there exists $0 < M_2 < M_3$ such that for any $-M_2 < M < 0$, the result \eqref{sub4_eq1} hold.

			\item $\frac{n}{2} + \frac{1}{4} \leq m < n$:
\\
In this case, we consider the condition \eqref{sec2_condition3}:
			\begin{align*}
				U - 1 - \frac{(n - 1)^2}{4n^2}\left[ \frac{nqU}{n(m - 1) + q}\right]^2 > 0.
			\end{align*}
			\begin{claim}\label{claim2}
				When $\frac{n}{2}+{1\over 4} \leq m < n$, we have
				\begin{align*}
					\Delta = 1 - \frac{(n - 1)^2}{n^2}\left[ \frac{nq}{n(m - 1) + q}\right]^2 > 0.
				\end{align*}
			\end{claim}
			\begin{proof}[Proof of claim \ref{claim2}]
			\begin{align*}
				\Delta &= 1 - (n - 1)^2\left[ \frac{q}{n(m - 1) + q}\right]^2\\
				&= \frac{(mn - n + q)^2 - (n - 1)^2q^2}{(mn - n + q)^2}\\
				&= \frac{-(n^2 - 2n)q^2 + 2(mn - n)q + (mn - n)^2}{(mn - n + q)^2}.
			\end{align*}
			Since $m - 1 < q < m - 1 + \frac{m}{n}$, we only need to check that
			\begin{numcases}{}
			\label{U1} -(n^2 - 2n)(m - 1)^2 + 2(mn - n)(m - 1) + (mn - n)^2 > 0,\\
			\label{U2} -(n^2 - 2n)\left( m - 1 + \frac{m}{n}\right)^2 + 2(mn - n)\left( m - 1 + \frac{m}{n}\right) + (mn - n)^2 > 0.	
			\end{numcases}
			In fact, we have
			\begin{align*}
				\eqref{U1} ~\Leftrightarrow~ & n(m - 1)\Bigg[-(n - 2)(m - 1) + 2(m - 1) + (mn - n) \Bigg] > 0\\
				\Leftrightarrow ~& 4n(m - 1)^2 > 0.\\
			\end{align*}
			On the other hand, we know
			\begin{equation}\label{inequa}
			\begin{aligned}
				\eqref{U2} \Leftrightarrow ~&-(n - 2)\left( mn - n + m\right)^2 + 2(mn - n)\left( mn - n + m\right) + n(mn - n)^2 > 0,\\
				\Leftrightarrow ~& 2m^2n^2+4n^2+2m^2+5m^2n-6mn(n+1) > 0.\\
				\end{aligned}
			\end{equation}
			Define
			\begin{align*}
				 f(m) &:= 2m^2n^2+4n^2+2m^2+5m^2n-6mn(n+1)\\
       &= m^2(2n^2+5n+2)-m(6n^2+6n)+4n^2.\\
			\end{align*}
          It is equivalent to prove
				\begin{align*}
      f(m) >0.\\
       \end{align*}
      Since $m \geq \frac{n}{2} + \frac{1}{4} > \frac{3n^2 + 3n}{2n^2+5n+2}$, we get that
      \begin{align*}
      	f(m) &\geq f\left( \frac{n}{2} + \frac{1}{4} \right)\\
      	&= \left( \frac{n}{2} + \frac{1}{4} \right)^2(2n^2+5n+2) - \left( \frac{n}{2} + \frac{1}{4} \right)(6n^2+6n)+4n^2\\
      	&= 0.5n^4 - 1.25 n^3 + 1.375n^2 - 0.6875 n + 0.125\\
      	&= 0.5(n - 0.5)^2(n^2 - 1.5n + 1).
      \end{align*}
So $f(m)>0$, i.e. \eqref{inequa} holds.
			\end{proof}
		\end{itemize}
So we know there exists $U_0 > 0$ such that condition \eqref{sec2_condition3} holds. At this time, we still choose the same $P, T$ as the first case, then \eqref{sec2_condition1} - \eqref{sec2_condition4} hold, but
\begin{align*}
\begin{cases}
		 \quad \alpha + p  - \frac{n(m - 1)\alpha U}{n(m - 1) + q} + \frac{(n - 1)^2}{2n^2}\Bigg[ {n\alpha\over n-1}- \frac{nmp}{n(m-1)+m} \Bigg]\cdot \frac{nqU}{n(m - 1) + q} > 0,\\	
		 \quad T+U-1 > 0.
	\end{cases}
\end{align*}
By Young's inequality,
		\begin{equation}\label{Young3}
			\begin{aligned}
				|M|v^{\alpha - 1}|\nabla v|^{m + q} \leq J v^{\alpha - 2}|\nabla v|^{2m} + \frac{1}{p + 2} \left(J \cdot \frac{p + 2}{p + 1}\right)^{- p - 1} |M|^{p + 2} v^{\alpha + p}|\nabla v|^q,
			\end{aligned}
		\end{equation}
		with $(\frac{p + 2}{p + 1} , p + 2)$, and
		\begin{equation*}
			\begin{aligned}
				|M| v^{\alpha + p}|\nabla v|^q \leq K v^{\alpha + 2p} + \frac{q}{m}\left(K\cdot \frac{m}{m - q}\right)^{- \frac{m - q}{q}}  |M|^{\frac{m}{q}} v^{\alpha + p - 1}|\nabla v|^m,
			\end{aligned}
		\end{equation*}
		with $(\frac{m}{m - q}, \frac{m}{q})$.
Therefore  there exists $0 < M_2 < M_3$ such that for any $-M_2 < M < 0$, the result \eqref{sub4_eq1} holds.
	\end{proof}

	\end{CJK}
\end{document}